\documentclass[10pt]{article}
\usepackage{graphicx}
\usepackage{latexsym}
\usepackage{amsmath}
\usepackage{amssymb}
\usepackage{amsthm}
\usepackage{multirow}
\usepackage{authblk}
\usepackage[spaces,hyphens]{url}
\usepackage{bm}
\usepackage[square,sort,comma,numbers]{natbib}
\usepackage{hyperref}
\usepackage{tikz}
\usepackage{multicol}
\usepackage{mathtools}

\def \N {{\mathbb N}}
\def \Z {{\mathbb Z}}
\def \R {{\mathbb R}}

\def \A {{\mathcal A}}
\def \D {{\mathcal D}}

\def \Gl {{\mathcal G}}

\def \J  {{\mathbf J}}

\def \al {{\alpha}}

\def \ga {{\gamma}}
\def \Om{{\Omega}}

\def \Ga {{\Gamma}}
\def \la {{\lambda}}

\newtheorem{theorem}{Theorem}[section]
\newtheorem{cor}[theorem]{Corollary}

\newtheorem{pro}[theorem]{Proposition}
\newtheorem{rem}[theorem]{Remark}

\newtheorem{ex}[theorem]{Example}

\usepackage{setspace}
\usepackage{natbib}
\usepackage[english]{babel}
\usepackage[euler]{textgreek}
\usepackage{stmaryrd}
\usepackage{mathtools}
\usepackage{pgf,tikz}
\usepackage{tikz}
\usepackage{tkz-graph}
\usepackage{array}
\tikzstyle{vertex}=[circle,draw, inner sep=0pt, minimum size=4pt] 
\newcommand{\vertex}{\node[vertex]}
\usetikzlibrary{decorations.markings}
\usepackage{lipsum}
\title{Generating Graphs of Finite Dihedral Groups} 
\author[1]{A. Satyanarayana Reddy}
\author[2]{Kavita Samant}
\affil[1,2]{Department of Mathematics\\Shiv Nadar Institution of Eminence, Delhi-NCR, India.}
\affil[2]{ks299@snu.edu.in;}
\affil[1]{satya.a@snu.edu.in}

\date{}

\begin{document}
\maketitle
\begin{abstract}
    For a group $G$, the generating graph $\Ga(G)$ is defined as the graph with the vertex set $G$, and any two distinct vertices of $\Ga(G)$ are adjacent if they generate $G$. In this paper, we study the generating graph of $D_n,$ where $D_n$ is a Dihedral group of order $2n$. We explore various graph theoretic properties, and determine complete spectrum of the adjacency and the Laplacian matrix of $\Ga(D_n)$. Moreover, we compute some distance and degree based topological indices of $\Ga(D_n)$.
\end{abstract}
\section{Introduction}
A group $G$ is said to be {\em two generated} if there exists a pair of elements that generates $G$. The {\em generating graph} of $G$, denoted by $\Ga(G),$ is a graph whose vertex set is $G$, and any two distinct vertices are adjacent if they generate $G$. Thus it is natural to consider only two generated groups; otherwise, the generating graphs are empty graphs. The fundamental idea behind generating graphs is the generation of groups by its two elements, which was first studied in a probabilistic perspective. The idea can be stated as, any two elements selected at random from a finite group $G$ generate $G$ with the probability $P(G)$, which is defined as
$$P(G)=\frac{|\{(x,y)\in G\times G:\langle x,y\rangle=G\}|}{|G|^2}.$$  For finite simple groups, many deep results had been studied using the probabilistic notion by Liebeck and Shalev in~\cite{Lieback}, Guralnick and Kantor in~\cite{Guralnick} and Andrea Lucchini and E. Detomi in~\cite{MR1987022} which can be equivalently stated as theorems that ensure that $\Ga(G)$ is a rich graph. Later concerning the study of the structure of two generated groups and their generating pairs, the concept of the generating graphs was defined in~\cite{MR2816431, MR2515391} by Andrea Lucchini and Attila Mar\'oti. These papers investigated various graph theoretic properties of the generating graphs of finite groups and proposed many questions. \par
In this article, our prime focus is on the generating graph of $D_n,$ denoted by $\Ga(D_n)$ with the vertex set $V$ and the edge set $E$. For our convenience, we write $\Ga_n$ instead of $\Ga(D_n)$ throughout this paper. We will start with some preliminary definitions and concepts of $D_n$ and its generating sets, which are covered in Section~\ref{sec:pre}. In Section~\ref{sec3:gan}, we will explore various graph theoretic properties for a better understanding of the structure of $\Ga_n.$ Since Dihedral groups are solvable, thus quite a lot is already known for the generating graphs of finite solvable groups related to the connectivity of graphs, given in~\cite{crestani2013generating, lucchini2017diameter, planar} etc. However for the sake of completion, we will give short proofs of these properties, and in addition, we will discuss more properties. \par
In Section~\ref{sec4:spec}, we will define the adjacency and the Laplacian matrix for $\Ga_n$, denoted by $A(n)$ and $L(n)$ respectively. We divide Section~\ref{sec4:spec} further into three subsections. In Subsection~\ref{sec4:AM}, we will show that corresponding to each $n\in \N,$ if we take $n_0$ as the square free part of $n,$ then
the eigenvalues of $A(n)$ are $0$, $\frac {1}{2}\left({\varphi(n)\pm \sqrt{\varphi(n)^2+4n\varphi(n)}}\right),$ and $\frac{n}{n_0}c_{n_0}(j)$ for $1\leq j\leq (n_0-1),$ where  $c_{n_0}(j)$ is the Ramanujan's sum. Thus, $\Ga_n$ is an integral graph depending on whether the two values $\frac {1}{2}\left({\varphi(n)\pm \sqrt{\varphi(n)^2+4n\varphi(n)}}\right)$ are integers or not. It can be seen that $\Ga_n$ is an integral graph if and only if $n=2^a$ where $a\in \N.$
Moreover, in Subsection~\ref{sec4:alter} we will give an alternative approach to compute the spectrum of $A(n).$ In Subsection~\ref{sec4:SLM}, we will show that the eigenvalues $L(n)$ are 0, $2\varphi(n)$, $n$, $n+\varphi(n)$, $2\varphi(n)-\frac{n}{n_0}c_{n_0}(j)$ where $1\leq j\leq (n_0-1).$ Hence for any $n\in \N,\, \Ga_n$ is a Laplacian integral graph. \par
In Section~\ref{sec5:top}, we will compute few topological indices of the graph $\Delta_n,$ {\it i.e.,} the graph obtained by removing the isolated vertices from $\Ga_n.$ Let $\mathcal {G}$ be the collection of all graphs. A map $I:\mathcal{G}\rightarrow \R$ is called a {\em topological index}, if $\mathcal{G}_1\cong \mathcal{G}_2$ implies that $I(\mathcal{G}_1)=I(\mathcal{G}_2).$ Several topological indices are based on the distances between pairs of vertices or the degree of the graph's vertices. These indices are used to study the structure of molecular graphs, which have wide applications in chemistry. Some well-known topological indices are the Wiener index, the Schultz index, the Zagreb index, etc. Over the recent year, several classes of graphs that are associated with groups of different topological indices have been studied~\cite{Nindex, Ncindex, zindex}, etc. We will study certain distance-degree based topological indices as a result of the aforementioned work.

\section{Preliminaries} \label{sec:pre}
In what follows, we present a few definitions, properties, and results related to Dihedral groups that we will use in rest of the paper. 
The Dihedral group $D_{n}$ of order $2n$ is defined by $$D_{n}=\langle r,s\,|\,r^n=s^2=1,\,sr=r^{-1}s\rangle \;\mbox{for any}\; n\geq 1.$$ We denote $R=\langle r \rangle,$ the group of rotations of  $D_n$ and $$\Omega_{1}=\{r^{i}\,:\, \gcd(n,i)=1\},  \Omega_{2}=\{sr^{i}\,:\, 0\leq i\leq (n-1)\}\; \mbox{and}\; 
\Omega_{3}=R\setminus \Omega_{1}.$$
So we have $$D_n=R\cup \Omega_2=\Omega_1\cup \Omega_2\cup \Omega_3.$$

It is known that $D_{n}$ is a two generated group as $D_n=\langle r,s\rangle.$ Let us collect all those pairs which generate $D_{n}$. For each $n\in \N,$ we define a set $Gen(n)$ by 
$$Gen(n)=\{(x,y)\in D_n\times D_n:\,\langle x,y\rangle=D_{n}\}.$$
Clearly it is a generating set of $D_{n}.$
It is exactly the set of all the generating pairs of $D_{n}$. First, note that any generator of $R$ with any reflection will generate $D_{n}$. Thus, we conclude that $(\Omega_{1}\times \Omega_{2})\cup (\Omega_{2}\times \Omega_{1})\subseteq Gen(n)$. However, $((\Omega_{3}\times \Omega_{i})\cup (\Omega_{i}\times \Omega_{3}))\cap Gen(n)=\emptyset,$ where $i=1,2.$ 
It is easy to see that $\langle sr^{i},sr^{j}\rangle=D_{n}$ if and only if $sr^{i} sr^{j}=r^{j-i}\in \Omega_{1}$ if and only if $\gcd(j-i,n)=1,$ where $0\leq i,j\leq n-1$. Thus, we can say that if we choose any two reflections randomly, they may not generate $D_{n}$. Now for each $sr^i\in \Om_2,$ where $0\leq i\leq n-1$, we define a set $$B_{i}=\{sr^{j}:\,j=a+i\, \text{ where } a\in U(n)\},$$ where $U(n)$ is the group of units of ring $\Z_{n}.$ Using this, we define an another set of the form 
$$A_{i}=\{sr^{i}\}\times B_{i}.$$  We observe that for each $i,$ the set $A_{i}\subseteq \Omega_{2}\times \Omega_{2}$ and  $\bigcup\limits_{i=0}^{n-1}A_{i}\subseteq Gen(n).$ Therefore,
$$Gen(n)=(\Omega_{1}\times \Omega_{2})\cup (\Omega_{2}\times \Omega_{1})\cup(\bigcup\limits_{i=0}^{n-1} A_{i}).$$ 
Notice that $|\Omega_{1}\times \Omega_{2}|=|\Omega_{2}\times \Omega_{1}|=n\varphi(n)$ and $|A_{i}|=\varphi(n)$ for each $i$. Thus, 
\begin{equation}\label{pre:sum}
|Gen(n)|=3n\varphi(n).
\end{equation}
Consequently, we have $P(D_n),$ the probability of two randomly chosen elements that generates $D_n$ is $\frac{3}{4}\left(\frac{\varphi(n)}{n}\right).$
By the definition, the generating graph of a group is an undirected graph. Thus, in $\Ga_n$ there is an edge between $a,b\in D_n$ if and only if $(a,b)\in Gen(n).$ In particular, the edge set of $\Ga_n$ is $E=\{\{x,y\}\,:\,(x,y)\in Gen(n)\}.$  Thus the number of edges in $\Ga_n$ is $\frac{3n\varphi(n)}{2}.$ Let $G$ be a two generated group and let $S(G)=\{x\in G\,:\,\exists \, z\in G,\langle x, z\rangle = G\}.$ Then, clearly $S(D_n)=\Omega_1\cup \Omega_2.$ Hence $\Ga_n$ is a disconnected graph, and the number of connected components equals $n-\varphi(n)+1.$ The subgraph of $\Ga_n$ induced by the vertices $S(D_n)$ is one connected component, and others are the isolated vertices that correspond to the members of $\Omega_3.$ We denote the graph obtained from $\Ga_n$ by removing the isolated vertices as $\Delta_n.$ For notational convenience, we adopt the notation $\Ga_X$ to denote the subgraph of $\Ga_n,$ which is induced by a subset $X$ of $V.$  Here, $\Delta_n=\Ga_{S(D_n)}.$ \\
If $N(x)$ denotes the neighbourhood of a vertex $x,$ then we have
\begin{equation}\label{eq:nbd}
 N(x)=\begin{dcases}
       \Omega_2 & \mbox{if $x\in \Omega_1,$}\\
       \Omega_1\cup B_i & \mbox{if $x=sr^i\in\Omega_2,$}\\
       \emptyset & \mbox{if $x\in \Omega_3.$}\\
      \end{dcases}
\end{equation}
Thus if $x\in V(\Ga_n),$ then $\deg(x)\in \{n,2\varphi(n),0\}.$ Observe that $2\varphi(n)=n$ if and only if $n=2^{a},$ where $a\geq 1$, thus in that case the vertex degrees of $\Ga_{2^a}$ are $0$ and $2^a$. In other words, $\Delta_{2^a}$ is $2^a$-regular graph. Let us see some examples below.
\begin{ex}
The graphs $\Ga_3,$ $\Ga_4$ and $\Ga_5$ were shown in Figure~\ref{fig:ga4}. 
\begin{figure}[!ht]
	\centering
	\begin{tabular}{ccc}
	\begin{tikzpicture}[scale=0.35]
	\draw [line width=0.5pt] (5,7.1)-- (5,-1);
	\draw [line width=0.5pt] (5,7.1)-- (1,3);
	\draw [line width=0.5pt] (5,7.1)-- (10,3);
	\draw [line width=0.5pt] (3,7.1)-- (5,-1);
	\draw [line width=0.5pt] (3,7.1)-- (1,3);
	\draw [line width=0.5pt] (3,7.1)-- (10,3);
 
	\draw [line width=0.5pt] (1,3)-- (5,-1);	
	\draw [line width=0.5pt] (10,3)-- (1,3);
	\draw [line width=0.5pt] (5,-1)-- (10,3);	
	\draw (0.5,3) node[anchor=north west] {$s$};
	\draw (10,3) node[anchor=north west] {$sr$};
	\draw (5,-1) node[anchor=north west] {$sr^{2}$};
	\draw (3,7.9) node[anchor=north west] {$r^{2}$};
	
	\draw (5,7.6) node[anchor=north west] {$r$};
	\draw (7,9) node[anchor=north west] {$1$};
	
	\begin{scriptsize}
	\draw [fill=black] (1,3) circle (1.5pt);
	\draw [fill=black] (10,3) circle (1.5pt);
	\draw [fill=black] (5,-1) circle (1.5pt);
	\draw [fill=black](3,7.1)circle (1.5pt);

	\draw [fill=black](5,7.1)circle (1.5pt);
	\draw [fill=black](7,9)circle (1.5pt);
	\end{scriptsize}
	\end{tikzpicture}
	
  &\hspace*{10mm}
	\begin{tikzpicture}[scale=0.4]
	\draw [line width=0.5pt] (10,6)-- (14,11);
	\draw [line width=0.5pt] (14,11)-- (8,11);
	\draw [line width=0.5pt] (8,11)-- (8,5);
	\draw [line width=0.5pt] (8,5)-- (14,5);
	\draw [line width=0.5pt] (14,5)-- (14,11);
	\draw [line width=0.5pt] (14,11)-- (10,10);
	\draw [line width=0.5pt] (10,10)-- (8,11);
	\draw [line width=0.5pt] (8,11)-- (10,6);
	\draw [line width=0.5pt] (10,6)-- (8,5);
	\draw [line width=0.5pt] (8,5)-- (10,10);
	\draw [line width=0.5pt] (10,10)-- (14,5);	
	\draw [line width=0.5pt] (14,5)-- (10,6);
	
	\draw (6.8,5) node[anchor=north west] {$sr^{2}$};
	\draw (14,11) node[anchor=north west] {$s$};
	\draw (6.5,11) node[anchor=north west] {$sr$};
	\draw (10,5.8) node[anchor=north west] {$r$};
	\draw (16,7.5) node[anchor=north west] {$r^{2}$};
	\draw (16,8.6) node[anchor=north west] {$1$};
	\draw (9.79,11.3) node[anchor=north west] {$r^{3}$};
	\draw (14,5) node[anchor=north west] {$sr^{3}$};
	
	\begin{scriptsize}
	\draw [fill=black](16,8.5)circle (1.5pt);
	\draw [fill=black] (8,5) circle (1.5pt);
	\draw [fill=black] (14,11) circle (1.5pt);
	\draw [fill=black] (8,11) circle (1.5pt);
	\draw [fill=black] (10,6) circle (1.5pt);
	\draw [fill=black](16,7.5)circle (1.5pt);
	\draw [fill=black] (10,10) circle (1.5pt);
	\draw [fill=black](14,5)circle (1.5pt);
	\end{scriptsize}
	\end{tikzpicture}
	 & \hspace*{10mm}
	\begin{tikzpicture}[scale=0.37]
	\draw [line width=0.5pt] (5,7.1)-- (1.2,1);
	\draw [line width=0.5pt] (5,7.1)-- (1,3);
	\draw [line width=0.5pt] (5,7.1)-- (10,3);
	\draw [line width=0.5pt] (5,7.1)-- (7,1);
	\draw [line width=0.5pt] (5,7.1)-- (5,-1);
	\draw [line width=0.5pt] (3,7.1)-- (1.2,1);
	\draw [line width=0.5pt] (3,7.1)-- (1,3);
	\draw [line width=0.5pt] (3,7.1)-- (10,3);
	\draw [line width=0.5pt] (3,7.1)-- (7,1);
	\draw [line width=0.5pt] (3,7.1)-- (5,-1);
	\draw [line width=0.5pt] (7,7.1)-- (1.2,1);	
	\draw [line width=0.5pt] (7,7.1)-- (1,3);
	\draw [line width=0.5pt] (7,7.1)-- (10,3);
	\draw [line width=0.5pt] (7,7.1)-- (7,1);
	\draw [line width=0.5pt] (7,7.1)-- (5,-1);
	\draw [line width=0.5pt] (9,7.1)-- (1.2,1);
    \draw [line width=0.5pt] (9,7.1)-- (1,3);
	\draw [line width=0.5pt] (9,7.1)-- (10,3);
	\draw [line width=0.5pt] (9,7.1)-- (7,1);
	\draw [line width=0.5pt] (9,7.1)-- (5,-1);
	\draw [line width=0.5pt] (1.2,1)-- (1,3);
	\draw [line width=0.5pt] (1.2,1)-- (10,3);
	\draw [line width=0.5pt] (1.2,1)-- (7,1);
	\draw [line width=0.5pt] (1.2,1)-- (5,-1);
	\draw [line width=0.5pt] (1,3)-- (5,-1);
	\draw [line width=0.5pt] (1,3)-- (7,1);
	\draw [line width=0.5pt] (1,3)-- (10,3);
	\draw [line width=0.5pt] (7,1)-- (5,-1);
	\draw [line width=0.5pt] (7,1)-- (10,3);
	\draw [line width=0.5pt] (5,-1)-- (10,3);

	\draw (0,1) node[anchor=north west] {$sr^{2}$};
	\draw (0,3) node[anchor=north west] {$s$};
	\draw (10,3) node[anchor=north west] {$sr$};
	\draw (7,1) node[anchor=north west] {$sr^{3}$};
	\draw (5,-1) node[anchor=north west] {$sr^{4}$};
	\draw (9,7.6) node[anchor=north west] {$r^{4}$};
	\draw (7,7.6) node[anchor=north west] {$r^{3}$};
	\draw (3,7.6) node[anchor=north west] {$r$};
	\draw (5,7.6) node[anchor=north west] {$r^{2}$};
	\draw (7,9) node[anchor=north west] {$1$};

	\begin{scriptsize}
	\draw [fill=black](1.2,1)circle (1.5pt);
	\draw [fill=black] (1,3) circle (1.5pt);
	\draw [fill=black] (10,3) circle (1.5pt);
	\draw [fill=black] (7,1) circle (1.5pt);
	\draw [fill=black] (5,-1) circle (1.5pt);
	\draw [fill=black](3,7.1)circle (1.5pt);
	\draw [fill=black] (7,7.1) circle (1.5pt);
	\draw [fill=black](5,7.1)circle (1.5pt);
	\draw [fill=black](9,7.1)circle (1.5pt);
	\draw [fill=black](7,9)circle (1.5pt);
	\end{scriptsize}
	\end{tikzpicture}
	\end{tabular}
\caption{Generating graph of $D_3$, Generating graph of $D_{4},$ and Generating graph of $D_{5}.$}
	\label{fig:ga4}	
\end{figure}
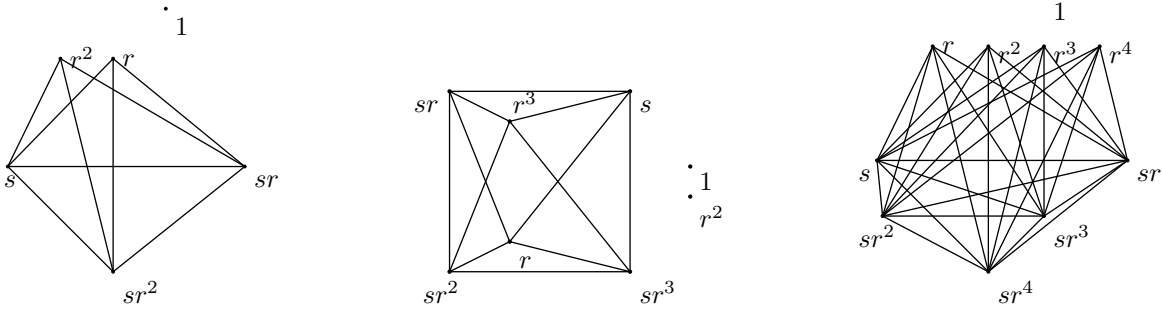
\end{ex}
\section{Few properties of $\Delta_n$ }\label{sec3:gan}
In this section, we will discuss many graph theoretic properties; for their standard definition, the reader can refer to~\cite{west2001introduction}. In order to study the properties of the graph $\Ga_n$ of $D_n$ it is sufficient to study the properties of $\Delta_n.$ Recall that the number of vertices of $\Delta_n$ is $n+\varphi(n)$ and the number of edges is $\frac{3n\varphi(n)}{2}.$ The degree of every vertex is either $n$ or $2\varphi(n).$ The graph $\Delta_n$ is not complete when $n\ge 3$ as $|\Omega_1|\ge 2.$ The girth of $\Delta_n$ is $3$ whenever $n\ge 2$ as $r,s,rs$ form a triangle, hence $\Delta_n$ is not bipartite. The following observations are also easy to see. 
\begin{enumerate}
 \item The graph $\Delta_n$ is regular if and only if $n=2^a$ for some $a\ge 1.$ 
 \item The graph $\Delta_n$ is Eulerian if and only if $n$ is even.
 \item The graph $\Delta_n$ is  Hamiltonian graph for $n\ge 2.$ We have $\varphi(n)$ rotations and $n$ reflections, and it is clear that for each $0\leq i \leq (n-1)$ there is an edge between and $sr^{i}$ and $sr^{i+1}.$ 
 Let $U(n)= \{a_1,a_2,\dots,a_{\varphi(n)}\}$ such that $a_1<a_2<\dots<a_{\varphi(n)}.$ Then $s-r^{a_1}-sr^{a_1}-sr^{a_1+1}-\dots-sr^{a_2-1}-r^{a_2}-sr^{a_2}-r^{a_2+1}-\dots -sr^{a_3-1}-r^{a_3}-sr^{a_3}-r^{a_3+1}-\dots-sr^{a_{\varphi(n)}-1}-r^{a_{\varphi(n)}}-sr^{a_{\varphi(n)}}-s$ is a  Hamiltonian path in $D_n.$
 \item It is easy to verify that $\Delta_n$ contains cycles starting at $s$  of all lengths from $3$ to $n+\varphi(n).$  
 \item  The graph $\Delta_n$ is planar if and only if $n\in\{1,2,3,4,6\}.$ This has been given in the paper~\cite{planar}. Here, we are giving a proof for the sake of completeness. It is easy to see that $D_4$ and $D_6$ are planar. If $n=5$ or $n\ge 7$, then $\varphi(n)\ge 4$  and one can see that $K_{3,3}$ as a subgraph in $\Delta_n$ as every element of $\Omega_1$ is adjacent to every   element of $\Omega_2.$ 
    \item The domination number of $\Delta_n$ is $2$ when $n(\neq 1)$ is not a prime. Otherwise, the domination number of $\Delta_n$ is 1. However the total domination number, denoted by $\delta_{t}(n)$ is 2 for any $n\geq 2$. Let us prove the argument. 
   First we assume that $n$ is not a prime, we can choose the set $\{r,s\}$ as a dominating set since any vertex in the set $S(D_n)\setminus \{r,s\}$ is either adjacent to $r$ or $s$ and it is clearly of the least cardinality. Moreover, the set $\{r,s\}$ can be seen as the total dominating set since $r$ and $s$ are also adjacent. For the case when $n$ is prime, $\Omega_{2}$ induced a clique of size $n$. Thus the set $\{s\}$ can be seen as the least dominating set, however $\delta_{t}(n)=2$ since the graphs are simple.
\item The clique number of $\Delta_n$ is equal to $p+1$ where $p$ is the least prime divisor of $n$. Let us see how it follows. Note that no two rotations are adjacent, thus finding the maximal size clique inside the subgraph $\Ga_{\Om_2}$ will be sufficient for us to prove the argument. First notice that $\Ga_{\Om_2}$ is a complete graph if and only if $n$ is prime. Then $\Omega_2\cup\{r\}$ forms a maximum size clique when $n$ is a prime. Thus, the clique number is $p+1$ in this case. However, the case when $n$ is not a prime, we can easily check that the set $\mathcal{W}=\{s, sr, sr^2, \ldots, sr^{p-1}\}$ forms a clique of size $p,$ where $p$ is the least prime divisor of $n.$ Now if we choose any $sr^j\in \Om_2 \setminus \mathcal{W},$ then obviously $j\geq p$ and we can write $j=tp+l$ for some $t\in \N$ and $0\leq l <p.$ Then for each $l,$ we have $sr^l\in \mathcal{W}$ and it satisfies $\gcd(tp+l-l,n)\neq 1.$ Therefore, it is easy to see that $\mathcal{W}$ induced a maximal size clique in $\Ga_{\Om_2}.$ Hence the set $\mathcal{W}\cup\{r\}$ is the maximal size clique in $\Delta_n.$ Thus, the clique number is $p+1.$
    \item The chromatic number is equal to clique number since $D_{n}$ has fitting height at most 2 (see Example 3.6 in \cite{conradsubgroup2}). The result follows from the Proposition~2.5 of \cite{lucchini2009clique}.
    \item Let  $n=p_1^{e_1}\ldots p_k^{e_k}$ where $p_i's$ are distinct primes and each $e_i\in \N.$ Suppose $p_1<p_2<\ldots< p_k.$ If $k>1,$ then the independence number of the graph $\Delta_n$ is $\frac{n}{p_1}.$  However when $k=1,$ it is exactly equal to $\varphi(n).$ 
    Let us prove the statement. First note that any independent set of $\Delta_n$ is either a subset of $\Omega_1$ or $\Om_2.$ It is clear that $\Om_1$ is itself an independent set of size $\varphi(n)$ in $\Delta_n.$
    Now we need to figure out the maximum size independent set in $\Delta_n$. We observed that corresponding to each prime divisor $p_i$ of $n$, the set $\mathcal{X}_i=\{s,sr^{p_i},sr^{2p_i},\ldots,sr^{(\frac{n}{p_i}-1)p_i}\}$ forms an independent set of size $\frac{n}{p_i}.$ Moreover if $p_1$ is the least prime, then $|\mathcal{X}_1|>|\mathcal{X}_j|$ for all $j>1.$ Now we claim that $\mathcal{X}_1$ is of maximal size independent set in $\Ga_{\Om_2}$. Let $sr^l\in \Om_2\setminus \mathcal{X}_1,$ then $l=tp_1+a$ with $\gcd(a,p_1)=1$ where $0\leq a\leq p_1-1$ and $0\leq t\leq \frac{n}{p_1}-1.$ Thus, $\gcd(a,n)=1$ and it implies that $(sr^{tp_1}, sr^l)$ is an adjacent pair in $\Delta_n,$ and hence cannot be added to the set $\mathcal{X}_{1}.$ Therefore, it is of maximal size independent set inside $\Om_2.$ So we have $\mathcal{X}_1$ and $\Om_1$ as the maximal size independent sets in $\Om_2$ and $\Om_1$ respectively. However our aim is to find which one is larger concerning given $n$. Note that when $k>1,$ $\varphi(n)\leq \frac{n}{p_1}$. This completes the proof.
\end{enumerate}
From the above discussion, the reader must have understood the structural properties of the graph $\Gamma_n.$ Now, one can see more such graph theoretical properties.
\section{Spectrum of the generating graph of $D_n$}\label{sec4:spec}
In this section, we will find the spectrum and energy of the adjacency and the Laplacian matrix of $\Ga_n$.  Further we denote $AE(n)$ and $LE(n)$ as the adjacency and Laplacian energy. Our aim is to find the spectrum of the adjacency matrix which we will discuss in  Subsection~\ref{sec4:AM}. Finding the spectrum of $L(n)$ is easy to compute once we sort out the spectrum of $A(n)$, so we will cover it in the last Subsection~\ref{sec4:SLM}.
\subsection{Spectrum of the Adjacency Matrix}\label{sec4:AM}
In this section, our goal is to find the spectrum of the adjacency matrix of $\Ga_n.$ We use some binary graph operations to give a nice form to the graph $\Ga_n.$ Moreover, we show that to find the spectrum of the adjacency matrix of $\Ga_n,$ it is sufficient to find the spectrum of the adjacency of the induced subgraph $\Ga_{\Om_2}.$ So, before diving into that we begin our section with some preliminary concepts of graph operations which we will use throughout this section. \par 
The {\em union} of two graphs $\mathcal{G}_1$ and $\mathcal G_2$, denoted by $\mathcal{G}_1\cup \mathcal{G}_2$ is the graph whose vertex set is $V(\mathcal{G}_1)\cup V(\mathcal{G}_2)$ and the edge set is $E(\mathcal{G}_1)\cup E(\mathcal{G}_2).$ The {\em join} of $\mathcal G_1$ and $\mathcal G_2$, denoted by $\mathcal{G}_1\vee \mathcal{G}_2$ is the graph obtained from $\mathcal{G}_1\cup \mathcal{G}_2$ by adding all the possible edges from the vertices of $\mathcal{G}_1$ to those of $\mathcal{G}_2.$ To understand these operations see the example below.
\vglue 1mm
\begin{table}[!ht]
\centering
 \begin{tabular}{|c|c|c|c|}
 \hline
\quad\quad $\mathcal{G}_1$ &\quad \quad$\mathcal{G}_2$ & \quad\quad$\mathcal{G}_1\cup \mathcal{G}_2$ &\quad \quad $\mathcal{G}_1\vee \mathcal{G}_2$\\
 \hline
\begin{tikzpicture}
    \vertex (1) at (1,2) {1};
    \vertex (2) at (2,2) {2};
    \path[-]
    (1) edge (2);
\end{tikzpicture} & 
 \begin{tikzpicture}
    \vertex (1) at (2,2) {3};
    \vertex (2) at (3,3) {4};
    \vertex (3) at (4,2) {5};
    \path[-]
    (1) edge (2)
    (2) edge (3)
    (3) edge (1);
\end{tikzpicture} & 
\begin{tikzpicture}
    \vertex (1) at (2,2) {1};
    \vertex (2) at (4,2) {2};
    \path[-]
    (2) edge (1);
\end{tikzpicture} 
 \begin{tikzpicture}
    \vertex (1) at (2,3) {3};
    \vertex (2) at (3,4) {4};
    \vertex (3) at (4,3) {5};
    \path[-]
    (1) edge (2)
    (3) edge (1)
    (3) edge (2);
\end{tikzpicture} &
\begin{tikzpicture}
    \vertex (1) at (2,2) {1};
    \vertex (2) at (4,2) {2};
    \vertex (3) at (2,3) {3};
    \vertex (4) at (3,4) {4};
    \vertex (5) at (4,3) {5};
    \path[-]
    (1) edge (2)
    (1) edge (3)
    (2) edge (3)
    (4) edge (5)
    (1) edge (4)
    (1) edge (5)
    (2) edge (4)
    (2) edge (5)
    (3) edge (5)
    (3) edge (4);
\end{tikzpicture}\\
\hline
\end{tabular}
\vglue 1mm
\caption{{Union and Join of graphs $\mathcal{G}_1$ and $\mathcal{G}_2.$}}\label{tab:oper1}
\end{table}
Let $A(\Gl)$ denotes the adjacency matrix, and let $L(\Gl)$ denotes the Laplacian matrix of a graph $\Gl.$ It is easy to see that the eigenvalues of the adjacency matrix of $\mathcal{G}_1\cup \mathcal{G}_2$ are the union of the eigenvalues of $A(\mathcal{G}_1)$ and  $A(\mathcal{G}_2).$ Similar holds for Laplacian in the case of the union of two graphs. We further assume that $\mathcal{G}_1$ is $r_1$-regular and $\mathcal{G}_2$ is $r_2$-regular. Let $\lambda_{1} \geq \lambda_{2} \geq \ldots \geq \lambda_{n_{1}}$ and $\ga_{1} \geq \ga_{2} \geq \ldots \geq \ga_{n_{2}}$ be the eigenvalues of $A(\mathcal{G}_1)$ and  $A(\mathcal{G}_2)$ respectively. Note that $\lambda_1=r_1$ and $\ga_1=r_2.$ Let $0=\la_1^L\leq\la_2^L\leq\ldots\leq\la_{n_1}^L$ and $0=\ga_1^L\leq \ga_2^L\leq \ldots\leq\ga_{n_2}^L$
 be the eigenvalues of $L(\mathcal{G}_1)$ and  $L(\mathcal{G}_2)$ respectively. The Laplacian and adjacency eigenvalues of the graph $\mathcal{G}_{1} \vee \mathcal{G}_{2}$ are given in the below table. \par
 The following table summarizes eigenvalues of both the matrices of the graphs $\mathcal{G}_{1} \cup \mathcal{G}_{2}$ and $\mathcal{G}_{1} \vee \mathcal{G}_{2}.$ Reader can refer to~\cite{spectra} for the detailed explanation.
\begin{table}[!ht]\label{Tspectra}
\centering
 \begin{tabular}{|c|c|c|}
 \hline
 Graphs & Eigenvalues of Adjacency &  Eigenvalues of Laplacian\\
 \hline
 $\mathcal{G}_{1}$ & $\lambda_{i}$ &   $\lambda_i^L$\\
 \hline
 $\mathcal{G}_{2}$ & $\ga_{j}$ &  $\ga_j^L$\\
 \hline
 $\mathcal{G}_{1} \cup \mathcal{G}_{2}$ & $\lambda_{i}$ &  $\lambda_i^L$\\
  & $\ga_{j}$ &  $\ga_j^L$\\
 \hline
 &&\\
 $\mathcal{G}_{1} \vee \mathcal{G}_{2}$ & $\lambda_{i};\;2\le i\le n_1$ & $\lambda_i^L+n_2;\,2\leq i\leq n_1$ \\
 & $\ga_{j};\;2\le j\le n_2$ &  $\ga_j^L+n_1;\,2\leq j\leq n_2$ \\
  & $\frac {r_{1}+r_{2}\pm \sqrt{(r_{1}-r_{2})^{2}+4n_{1}n_{2}}}{2}$  &  $0,\,n_1+n_2$\\
  &&\\
 \hline
  \end{tabular}
\caption{Spectral properties of the graphs obtained by the union and join of disjoint graphs.}
\label{tab:2}
\end{table}
\newpage
From the above discussion on graph operations and the spectrum of the associated graphs, the generating graph $\Ga_n$ can easily be seen as
\begin{equation}\label{Eq:exp}
\Ga_n=(\Ga_{\Om_1}\vee \Ga_{\Om_2})\cup \Ga_{\Om_3}.
\end{equation}
Note that $\Ga_{\Om_2}$ is a $\varphi(n)$-regular graph and $\Ga_{\Om_1},$ $\Ga_{\Om_3}$ are the empty graphs. Let us relabel the vertices of $\Ga_n$ in the rows and columns of $A(n)$ as per the expression $\Ga_n=(\Ga_{\Om_2}\vee \Ga_{\Om_1})\cup \Ga_{\Om_3}$ we have
$$A(n)=\begin{tabular}{c|ccc}
         & $\Om_2$ & $\Om_1$ & $\Om_3$\\
         \hline
         $\Om_2$ & $A_{11}$ & $A_{12}$ & $A_{13}$\\
         $\Om_1$ & $A_{21}$ & $A_{22}$ & $A_{23}$\\
         $\Om_3$ & $A_{31}$ & $A_{32}$ & $A_{33}$\\
       \end{tabular},$$ where $A_{11},A_{22}, A_{33}$ are the adjacency matrices of $\Ga_{\Om_2}, \Ga_{\Om_1}, \Ga_{\Om_3}$ respectively.
Let $O_{m\times n}$ and $J_{m\times n}$ be the matrices of size  $m\times n$ with all entries $0$ and $1$ respectively. Then from  Equation~\ref{eq:nbd}, we have
$$A(n)=\begin{bmatrix}
        A_{11} & J_{n\times \varphi(n)} & O_{n\times (n-\varphi(n))}\\
        J_{\varphi(n)\times n } & O_{\varphi(n)\times \varphi(n)} & O_{\varphi(n)\times (n-\varphi(n))}\\
        O_{(n-\varphi(n))\times n} & O_{(n-\varphi(n))\times\varphi(n)} & O_{(n-\varphi(n))\times (n-\varphi(n))}
       \end{bmatrix}.$$\\
Clearly, the rank of the matrix $A(n)$ is 1+rank of $A_{11}.$ The number of nonzero eigenvalues of $A(n)$ is equal to the rank of $A(n)$ since $A(n)$ is a symmetric matrix.  In fact, the nonzero eigenvalues of $A(n)$ are exactly  $\frac {\varphi(n)\pm \sqrt{\varphi(n)^2+4n\varphi(n)}}{2}$ and nonzero values among $\la_2,\la_3,\ldots,\la_n,$ where $\varphi(n)=\lambda_{1} \geq \lambda_{2} \geq \ldots \geq \lambda_{n}$ are eigenvalues of $\Ga_{\Om_2}$ (or $A_{11}$).
Now our goal is to find $\la_2,\la_3,\ldots,\la_n.$ For that, we define an {equivalence} relation on $\Om_2$ as $sr^i\sim sr^j$ if $B_i=B_j.$ It is easy to see that the equivalence class of $sr^i,$ where $i\in \{0,1,2,\ldots,n-1\}$ denoted by $[sr^i],$ is as follows $$[sr^i]=\{sr^i,sr^{i+n_0},sr^{i+2n_0},\ldots, sr^{i+(\frac{n}{n_0}-1)n_0}\}.$$ In particular, $$[s]=  \{s,sr^{n_0},sr^{2n_0},\ldots, sr^{(\frac{n}{n_0}-1)n_0}\}.$$ 
Since equivalence classes are either identical or disjoint, therefore we have $n_0$ disjoint equivalence classes\\ $\{[s],[sr],[sr^2],\ldots,[sr^{n_0-1}]\}$. Also, it is clear that each element of 
$[sr^i]$ has same neighbourhood set $B_i$ in the graph $\Ga_{\Om_2}.$ Thus, we can conclude that the rank of $A_{11}$ is $n_0$ and the rank of $A(n)$ is $n_0+1.$ Note that $|B_i|=\varphi(n)$ for every $i$ and by the definition of the relation $\sim$ one can easily see that each $B_i$ is a disjoint union of some equivalence classes of $\Om_2.$ Since each class has size $\frac{n}{n_0}$, thus there are exactly $\varphi(n_0)$ classes inside each $B_i.$ However, we need to figure out which are those.\par
Recall that corresponding to each $sr^i\in \Om_2$, we have defined $B_i=\{sr^{j}\in \Om_2:j=i+a, \text { where } a\in U(n)\}.$ We know that $\gcd(n,i-j)=1$ if and only if $\gcd(n_0,i-j)=1$. Using these facts, for each $i$ where $0\leq i\leq (n_0-1)$, the set $B_i$ can be seen as
\begin{equation}\label{sec4:Bi} B_i=\bigcup\limits_{\substack{1\leq j\leq n_0\\ \gcd(j,n_0)=1}}[sr^{i+j}].
\end{equation}\par
The above discussion help us to see the adjacency matrix of $\Ga_{\Om_2}$ concerning the defined equivalence classes in a very compact form. So we relabel the vertices in rows and columns as per the set of classes $\{[s],[sr],[sr^2],\ldots,[sr^{n_0-1}]\},$ and then we can construct $A_{11}$ as 

$$A_{11}=\begin{tabular}{c|cccc}
         & $[s]$ & $[sr]$& $\ldots$& $[sr^{n_0-1}]$\\
         \hline
         
         $[s]$ &$B_{11}$&$B_{12}$&$\ldots$& $B_{1n_0}$\\
         
         $[sr]$  &$B_{21}$&$B_{22}$&$\ldots$&$B_{2n_0}$\\
         
         $\vdots$  &$\vdots$&$\vdots$&$\ddots$&$\vdots$\\
         
         $[sr^{{n_0}-1}]$ & $B_{{n_0}1}$&$B_{{n_0}2}$&$\ldots$&$B_{{n_0}n_0}$\\ \\
       \end{tabular},$$
where $B_{ij}$ is given by
\begin{center}
$B_{ij}=\begin{dcases}
 		J_{\frac{n}{n_0}\times \frac{n}{n_0}}& \quad sr^i\in B_j;\\
 	     O_{\frac{n}{n_0}\times \frac{n}{n_0}}&\quad\text{otherwise.}\\
 		\end{dcases}$
   \end{center}
Note that these graphs do not contain self loops, thus for each $i$, the equivalence class $[sr^i]$ is not inside its neighbourhood set $B_i.$ Thus, $B_{ii}=0.$ When $i\ne j,$ $B_{ij}$ is $J_{\frac{n}{n_0}\times \frac{n}{n_0}}$ or $O_{\frac{n}{n_0}\times \frac{n}{n_0}}$ depending on the corresponding classes $[sr^i]$ and $[sr^j].$ Let us determine the structure of $A_{11}$ in the following examples.
 \begin{ex}
      Let $n=p,$ where $p$ is a prime. Then the equivalence classes of $\Om_2$ are  $[sr^i]=\{sr^i\}$ where $0\leq i\leq p-1.$ Thus we have
      \begin{center}
 		$A_{11}=J-I,$
 	\end{center}
 where $J-I$ is the block of order $p\times p.$ Clearly, $A_{11}$ is the adjacency matrix of the complete graph $K_p.$
     \end{ex}
     \begin{ex}
     Let $n=2^{\alpha},$ where $\alpha\in \N.$ The set of disjoint equivalence classes of $\Om_2$ is $\{[s],$ $[sr]\}.$ Note that $[s]=B_0$ and $[sr]=B_1$. Then the adjacency matrix is given by 
     \begin{center}
$A_{11}=~\left[\begin{matrix}0&J\\J&0\end{matrix}\right],$
 	\end{center}
  where $J$ and $0$  are the blocks of order $2^{\al-1}\times 2^{\al-1}.$ Clearly, $A_{11}$ is the adjacency matrix of the complete bipartite graph $K_{2^{\al-1},2^{\al-1}}$
    \end{ex}
 Now we will prove that the equivalence relation $\sim$ on $\Om_2$ gives an equitable partition; for that, let us see an overview of the concept of equitable partition.
    Let $\mathcal{G}$ be a graph with the vertex set $V$, and the edge set $E$. Let $\Pi=\{U_{1},U_{2},\ldots,U_{t}\}$ be a partition of $V$. We call each $U_i$ a cell of $\Pi$, and we denote $\deg(x,U_j)=$ the number of vertices in the cell $U_j$ adjacent to the vertex $x.$ We say $\Pi$ is an equitable partition of $V$ when for each pair of cells $(U_{i},U_{j})$(not necessarily distinct) $\deg(x,U_{j})=\deg(y,U_{j})=b_{ij},$ where $b_{ij}$ is a constant for every $x,y\in U_{i}.$

 The graph with $t$ cells of $\Pi$ as its vertices and $b_{ij}$ arcs from the $ith$ to the $jth$ cells of $\Pi$ are called the quotient of $\mathcal{G}$ over $\Pi$, and it is denoted by $\mathcal{G}/\Pi.$ Thus, the adjacency matrix of the quotient graph is given by $$A(\mathcal{G}/\Pi)=[b_{ij}]$$ where each $b_{ij}$ corresponds to the cells $U_i$ and $U_j.$ Now let us recall a result which is a consequence of equitable partitioning. To explore more about equitable partition one can refer to~\cite{godsil}.
\begin{theorem}[Godsil and Royle,~\cite{godsil}]\label{factor}
  If $\Pi$ is an equitable partition of a graph $\mathcal{G},$ then the characteristic polynomial of $A(\mathcal{G}/\Pi)$ divides the characteristic polynomial of $A(\mathcal{G}).$
\end{theorem}
Using the concept of equitable partition on $\Ga_{\Om_2},$ we have the following result.
\begin{theorem}
    Let $\Ga_{\Om_2}$ be the subgraph of $\Ga_n$ induced by $\Om_2.$ Then the relation $\sim$ on $\Om_2$ gives an equitable partition. 
\end{theorem}
 \begin{proof}
Let $\Pi=\{[s],[sr],[sr^2],\ldots,[sr^{n_0-1}]\}$ be a partition of $\Om_2.$ We discussed earlier that no two members in a class are adjacent. However, each $x\in [sr^i]$ is adjacent to the elements of $B_i,$ and each $B_i$ can be seen as the union of some equivalence classes of $\Om_2$.
 Since $|[sr^i]|=\frac{n}{n_0}$ for every $i,$ thus for any $x,y\in [sr^i]$
 $$\deg(x,[sr^j])=\deg(y,[sr^j])=\frac{n}{n_0}$$ for every cell $[sr^j]\subseteq B_i.$ For the rest of the cells $[sr^j],$ which are not in $B_i$
$$\deg(x,[sr^j])=\deg(y,[sr^j])=0.$$ Thus, $\Pi$ is an equitable partition of the graph $\Ga_{\Om_2}$.

 \end{proof}
 Now we will try to figure out the adjacency matrix of the quotient graph $\Ga_{\Om_2}/\Pi.$ Let ${\widetilde{A}_{11}}=A(\Ga_{\Om_2}/\Pi)$. Then finding the eigenvalues of ${\widetilde{A}_{11}}$ will be sufficient to find the whole spectrum of $\Ga_{\Om_2}.$ 
 Using the expression of $B_i$ given in Equation~\ref{sec4:Bi}, we can write ${\widetilde{A}_{11}}$ in the following form
\begin{equation}\label{quotient}
{\widetilde{A}_{11}}=\begin{tabular}{c|cccc}
         & $[s]$ & $[sr]$& $\ldots$& $[sr^{{n_0}-1}]$\\
         \hline
         $[s]$ &$a_0$&$a_1$&$\ldots$& $a_{{n_0}-1}$\\
         $[sr]$  &$a_{{n_0}-1}$&$a_0$&$\ldots$&$a_{{n_0}-2}$\\
         $\vdots$  &$\vdots$&$\vdots$&$\ddots$&$\vdots$\\
         $[sr^{{n_0}-1}]$ & $a_1$&$a_2$&$\ldots$&$a_0$\\ \\
       \end{tabular}
       \end{equation} 
       where for $0\leq i\leq ({n_0}-1),$ $a_i$ is given by 
$$a_i=\begin{dcases}
 		\frac{n}{n_0} & \quad \text{if }\gcd(i,n_0)=1;\\
 	     0&\quad\text{otherwise.}
 		\end{dcases}$$
Now one can observe that the matrix ${\widetilde{A}_{11}}$ is a circulant matrix with the first row vector $$[a_0\quad a_1\quad \ldots\quad a_{{n_0}-1}],$$ and the associated polynomial $q_{n}(t)$ is given by $$q_{n}(t)=\sum\limits_{i=0}^{n_0-1} a_it^i=\left(\frac{n}{n_0}\right)\sum\limits_{\substack{0\leq i\leq ({n_0}-1)\\ \gcd(i,n_0)=1}}t^i.$$
~\\
Then, we can express  ${\widetilde{A}_{11}}=q_{n}(P),$ where $P$ is a permutation matrix 
 with $P^{n_0}=I$ and $I$ is the identity matrix. 
Thus we can conclude that the eigenvalues of ${\widetilde{A}_{11}}$ are $q_{n}(\zeta_{n_0}^j)=\frac{n}{n_0}c_{n_0}(j)$ where $j\in \{0,1,\ldots,n_0-1\}$ and $\zeta_{n_0}$ is a primitive $n_0$-{th} root of unity. Moreover, we can also write each $c_{n_0}(j)$ in terms of the M\"{o}bius function, denoted by $\mu(\cdot).$ 
We can express the sum of the $j$th power of primitive $n_0$-th roots of unity as follows.
 $$c_{n_0}(j)=\mu(d)\frac{\varphi(n_0)}{\varphi(d)},$$
 where $d=\frac{n_0}{\gcd(n_0,j)}.$ 
 Thus, $$\frac{n}{n_0}c_{n_0}(j)=\mu(d)\frac{\varphi(n)}{\varphi(d)}=(-1)^{k_d}\frac{\varphi(n)}{\varphi(d)},$$
 here we are using the fact that $n_0$ is a square free integer, therefore $\mu(d)=(-1)^{k_d}$ where $k_d$ is the number of distinct prime factors of $d.$ Hence the eigenvalues of ${\widetilde{A}_{11}}$ are all integers. 
 In fact, one can note that there are exactly $\tau(n_0)$ distinct eigenvalues of $\widetilde{A}_{11}.$ Since $0\leq j\leq n_0-1,$ therefore each eigenvalue has multiplicity $\varphi(d).$ Hence $\sum\limits_{d/{n_0}}\varphi(d)=n_0.$ 
Now we can conclude the following results.
 \begin{theorem}
    Let $\Ga_n$ be the generating graph of $D_n.$ Then the characteristic polynomial of ${\widetilde{A}_{11}}$ is given by $$C_{{\widetilde{A}_{11}}}(t)=\prod\limits_{d/n_0}\left(t-(-1)^{k_d}\frac{\varphi(n)}{\varphi(d)}\right)^{\varphi(d)},$$ where $k_d$ is the number of distinct prime factors of $d.$ 
\end{theorem}
 \begin{cor}\label{Cor:spec2}
      The characteristic polynomial of the adjacency matrix $A_{11}$ of $\Ga_{\Om_2}$ is given by
      $$C_{A_{11}}(t)=t^{(n-n_0)}C_{\widetilde{A}_{11}}(t),$$ where $C_{\widetilde{A}_{11}}(t)$ is the characteristic polynomial of ${\widetilde{A}_{11}}.$
 \end{cor}
  \begin{proof}
     Note that ${\widetilde{A}_{11}}$ is invertible and $A_{11}$ is symmetric, therefore by using Theorem \ref{factor} we get the required result.
 \end{proof}
 Let us illustrate the above results in the following example.
\begin{ex}\label{n45}
 Let $n=45,$ then $n_0=15$. The subset $\Om_2$ of $V(\Ga_{45})$ is given by $$\Om_2=\{s,sr,\ldots,sr^{44}\}.$$ Note that there are exactly $n_0$ equivalence classes $[sr^i]$ and thus $n_0$ distinct $B_i's$ which varies from $0$ to $n_0-1.$ Each $B_i$ 
 can be expressed as $$B_i=\bigcup\limits_{\substack{1\leq j\leq 15\\ \gcd(j,15)=1}}[sr^{i+j}].$$ In particular, $$B_0=[sr]\cup[sr^2]\cup[sr^4]\cup[sr^7]\cup[sr^8]\cup[sr^{11}]\cup[sr^{13}]\cup [sr^{14}].$$ Since $[s]$ has the neighbourhood set $B_0,$ therefore the class $[s]$ will be adjacent to the classes which are inside $B_0.$ Thus, we get the first row of the matrix ${\widetilde{A}_{11}}$ as follows.
 $$[0\quad 3\quad 3\quad 0\quad 3\quad 0\quad 0\quad 3\quad 3\quad 0\quad 0\quad 3\quad 0\quad 3\quad 3].$$  Hence the eigenvalues of ${\widetilde{A}_{11}}$ are $24,-12,-6,3$ with multiplicity $1,2,4,8$ respectively.
 \end{ex}
 Now we will determine the adjacency energy of $\Ga_n.$ 
\begin{cor}
    For $n\in \N,$ the energy of $A(n)$ is given by $$AE(n)=\varphi(n)(2^k-1)+\sqrt{\varphi(n)^2+4n\varphi(n)}$$ where $k$ is the number of distinct prime divisors of $n.$ 
\end{cor}
\begin{proof}
    Let $1=d_1<d_2<\ldots<d_{2^k}=n_0$ be the divisors of $n_0.$ From the corollary~\ref{Cor:spec2}, the spectrum of $A(n)$ is given by $$\sigma_{A(n)}=
    \begin{pmatrix}
   0&\mu(d_2)\frac{\varphi(n)}{\varphi(d_2)}&
   \mu(d_3)\frac{\varphi(n)}{\varphi(d_3)}&\dots
   &\mu(d_{2^k})\frac{\varphi(n)}{\varphi(d_{2^k})}&\lambda_1&\lambda_2\\
   2n-(n_0+1)&\varphi(d_2)&\varphi(d_3)&\dots&\varphi(d_{2^k})&1&1
\end{pmatrix},
$$
where $\lambda_1=\frac {1}{2}\left(\varphi(n)+\sqrt{\varphi(n)^2+4n\varphi(n)}\right),\,\lambda_2=\frac{1}{2}\left(\varphi(n)- \sqrt{\varphi(n)^2+4n\varphi(n)}\right).$ Since $\lambda_1\lambda_2=-n\varphi(n)<0,$ thus $\lambda_2<0.$ Therefore $|\lambda_1|+|\lambda_2|=|\lambda_1-\lambda_2|=\sqrt{\varphi(n)^2+4n\varphi(n)}.$ Hence 
\begin{equation}\label{eq:AM}
AE(n)=\sum\limits_{i=2}^{\tau(n_0)}\left|(-1)^{k_{d_i}}\frac{\varphi(n)}{\varphi(d_i)}\right|\varphi(d_i)+\sqrt{(\varphi(n))^2+4n\varphi(n)}
\end{equation}
where $k_{d_i}$ represents the number of prime factors of $d_i.$
On simplifying Equation~\ref{eq:AM},  we have
\begin{align*}AE(n)&=\sum\limits_{i=2}^{2^k}\frac{\varphi(n)}{\varphi(d_i)}\varphi(d_i)+\sqrt{(\varphi(n))^2+4n\varphi(n)}\\
&=\varphi(n)(2^k-1)+\sqrt{(\varphi(n))^2+4n\varphi(n)}.
\end{align*}
Hence the result follows.
\end{proof}
For a given $n,$ the following result is concluded on some simple observations which gives a relation between the quotient matrices $A_{11}'$ and $A_{11}$ that corresponds to the graphs $\Ga_{n_0}$ and $\Ga_n$ respectively.  
\begin{theorem}
   For any $n$ such that $n\neq n_0$, if $\la$ is any non zero eigenvalue of the submatrix $A_{11}'$ of $A(n_0)$, then $(\frac{n}{n_0})\la$ is an eigenvalue of the submatrix $A_{11}$ of $A(n)$. In fact, $\la$ and $\frac{n}{n_0}\la$ are of the same multiplicities in their respective matrices. Moreover, $rank\, A(n)= rank\,A(n_0)$ but the nullities of $A(n)$ and $A(n_0)$ differs.
   \end{theorem}
   \begin{proof}
       In the case of $\Ga_{n_0,}$ each class $[sr^i]=\{sr^i\}$ where $0\leq i\leq n_0-1.$ Thus, the submatrix ${\widetilde{A}_{11}}'=A_{11}'$ of $A(n_0)$. However, for $\Ga_n$ where $n\neq n_0,$ the quotient matrix ${\widetilde{A}_{11}}\neq A_{11}.$ In addition, we can easily check that ${\widetilde{A}_{11}}=\frac{n}{n_0}{\widetilde{A}_{11}}'=\frac{n}{n_0}{\widetilde{A}_{11}}$. Therefore, if $\la$ is any non zero eigenvalue of the submatrix $A_{11}'$ of $A(n_0)$, then $(\frac{n}{n_0})\la$ is an eigenvalue of the matrix $A_{11}$ of $A(n)$. Clearly $\la$ and $\frac{n}{n_0}\la$ are of the same multiplicities in their respective matrices. Thus, $rank\, A(n)= rank\,A(n_0).$ However $n\neq n_0,$ therefore the nullities differ.
   \end{proof}
We conclude this section by one more simple observation. In the next subsection, we will give an alternative approach to find the eigenvalues of ${\widetilde{A}_{11}}.$
\begin{rem}
Recall that a connected graph $\Gl$ is primitive if and only if the graph contains cycles of odd lengths. The exponent of the graph is denoted by
$exp(\Gl)$. We have observed that for any $n\in \N,$ the graph $\Delta_n$ is connected and contains cycles of all the odd lengths for $n\geq 2.$ Using the above fact, we conclude that $\Delta_n$ is primitive and $exp(\Delta_n)=2.$
\end{rem}
\subsection{An alternative approach}\label{sec4:alter}
In the previous subsection, for any $n\in \N$ we have computed the eigenvalues of the matrix ${\widetilde{A}_{11}}.$ Now we give an alternative way to see the spectrum of ${\widetilde{A}_{11}}.$
We basically give an algorithm to reshuffle the labeling of rows and columns of ${\widetilde{A}_{11}}$ which help us to find a recurrence relation on the matrix $J-I$ where $J$ is the matrix of all entries 1's and $I$ is the identity matrix of order depending on the prime factors of $n.$ To understand the algorithm we first give an example of a Kronecker product, and then we will compute the spectrum.  \\ Let $A$ and $B$ be the square matrices of the form $J-I,$ which are of order $a\times a$ and $b\times b$ respectively. Then the Kronecker product $A\otimes B$ is of order $ab\times ab$ given by
$${\displaystyle {A} \otimes {B} ={\begin{bmatrix}0&{B} &\cdots &{B} \\
{B} &0&\cdots &B\\\vdots &\vdots&\ddots &\vdots \\{B} &{B}&\cdots &0\end{bmatrix}}}.$$ 
It is known that eigenvalues of the Kronecker product of two square matrices is all the possible product of their eigenvalues. Therefore, the eigenvalues of $A$ are $-1$ and $a$ with multiplicity $(a-1)$ and 1 respectively. Also, the eigenvalues of $B$ are $-1$ and $b$ with multiplicity $(b-1)$ and $1$ respectively. Hence the spectrum of $A\otimes B$ is given by 
$$\sigma=
    \begin{pmatrix}
        1& -a& -b&ab\\
        (a-1)(b-1)&(b-1)&(a-1)&1
    \end{pmatrix}.$$
Now we will find an algorithm which help us to reshuffle the rows and columns in such a way so that we will get ${\widetilde{A}_{11}}$ in the Kronecker product of smaller order matrices.\par
Let $n=p_1^{e_1}p_2^{e_2}\ldots p_k^{e_k},$ where $p_i's$ are the distinct primes and $e_i\in \N.$ Suppose $p_1<p_2<\dots <p_k$ and we have $n_0=p_1p_2\ldots p_k.$ Previously, the matrix ${\widetilde{A}_{11}}$ is labeled in the sequence of cells $[s],[sr],\ldots,[sr^{n_0-1}]$ in which each cell has been treated as a vertex, then we follow the following steps to get the relabeling of ${\widetilde{A}_{11}}.$
\begin{description}
\item[Step 1.] We first divide the set $\{[s],[sr],\ldots,[sr^{n_0-1}]\}$ into $p_1$ subsets of size $\frac{n_0}{p_1}$, which are defined as
$$V_{i_1}=\{[sr^{p_1t+i_1}]:0\leq t\leq \frac{n_0}{p_1}-1\},$$
where $0\leq i_1\leq (p_1-1).$ Note that each $V_{i_1}$ can be seen as an independent set.
\item[Step 2.] The labeling of rows and columns of ${\widetilde{A}_{11}}$ will now follow the sequence $V_0$,$V_1,\ldots, V_{p_1-1}.$ However inside each subset $V_i,$ the labeling will be done by using the next step.
\item[Step 3.] We subdivide each $V_{i_1}$ into $p_k$ subsets, of sizes $\frac{|V_{i_1}|}{p_k}$ and define each subset as
$$V_{i_1 i_2}=\{[sr^l]\in V_{i_1}:l={p_kt+i_2}\text{ and }0\leq t\leq \frac{n_0}{p_1p_k}-1\},$$ where $0\leq i_2\leq p_{k}-1.$
\item[Step 4.] Now the labeling of ${\widetilde{A}_{11}},$ corresponding to  each subsets $V_{i_1}$ will follow the sequence $V_{i_10}$, $V_{i_11},$
\ldots, $V_{i_1p_{k-1}}.$
\item[Step 5.] Again we subdivide each $V_{i_1i_2}$ into $p_{k-1}$ subsets, of sizes $\frac{|V_{i_1 i_2}|}{p_{k-1}}.$ We define each subset as
$$V_{i_1i_2i_3}=\{[sr^l]\in V_{i_1i_2}:l={p_{k-1}t+i_3}\text{ and }0\leq t\leq \frac{n_0}{p_1p_kp_{k-1}}-1\},$$ where $0\leq i_3\leq p_{k-1}-1.$
\item[Step 6.] Now the labeling of ${\widetilde{A}_{11}},\,$ corresponding to each $V_{i_1i_2i_3}$ will follow the sequence $V_{i_1i_2i_30}$,\,$V_{i_1i_2i_31}$,\,\ldots, \,$V_{i_1i_2i_3(p_{k-1}-1)}.$
\item [Step 7.]We will continue dividing each subsequent subsets until we get subsets say $V_{\Lambda}$ of size $p_2,$ where $\Lambda=\{i_1i_2\ldots i_{(k-1)}\}$ is the indexing set.
\item [Step 8.] We label the rows and columns corresponding to each subset $V_\Lambda$ in the sequence $[sr^{p_2t}],$$[sr^{p_2t+1}]$,\,\ldots, $[sr^{p_2t+(p_2-1)}],$ and thus we get the final labeling of rows and columns of $\widetilde{A}_{11}$.
\end{description}
Using the above algorithm, the matrix ${\widetilde{A}_{11}}$ can be seen in the following block form, where in each step  because of the subdivision we will get a block matrix. In that way, we will get $(k-1)$ tuples of block matrices $(A_{p_k},A_{p_{k-1}},\ldots, A_{p_2}),$ and then ${\widetilde{A}_{11}}$ comes in the following form. Here we consider two cases with respect to the number of prime factors of $n$.
\begin{enumerate}
    \item When $k=1$, $n=p_{1}^{e_{1}}.$ Then we have
  $${\widetilde{A}_{11}}=\frac{n}{n_{0}}(J-I),$$ where $J$ and $I$ are the matrices of order $p_{1}\times p_{1}.$
\item When $k>1,$ we get $(k-1)$ tuples of block matrices $(A_{p_{k}},A_{p_{k-1}},\ldots,A_{p_{2}}).$\\  From Step(1), the matrix ${\widetilde{A}_{11}}$ can be seen in the following block form.
$${\widetilde{A}_{11}}=\frac{n}{n_0}\begin{pmatrix}0&A_{p_{k}}&\cdots&A_{p_{k}}\\A_{p_{k}}&0&\cdots&A_{p_{k}}\\ \vdots&\vdots&\ddots&\vdots\\A_{p_{k}}&A_{p_{k}}&\cdots &0\end{pmatrix},$$
 where $A_{p_{k}}$ and $0$ are the blocks of size $\frac {n_{0}}{p_{1}}\times \frac {n_{0}}{p_{1}}.$ Moreover, $A_{p_{k}}$ is repeating $(p_{1}-1)$ times in each row since each block corresponds to a pair of subsets $(V_{i_1},V_{j_1})$ for $i_1\neq j_1$ where $i_1,j_1\in \{0,1,\ldots,p_1-1\}.$ 
 
 From Step(2), we get a block $A_{p_{k-1}}$ inside each matrix $A_{p_{k}},$ which is repeating $(p_{k}-1)$ times in each row. Each $A_{p_{k-1}}$ corresponds to a pair of subsets $(V_{i_1i_2},V_{j_1j_2})$ for $i_2\neq j_2$ and $i_2,j_2\in \{0,1,\ldots,p_k-1\}.$ 
 Note that when $i_2=j_2$, the block corresponding to $(V_{i_1i_2},V_{j_1j_2})$ is the zero matrix.
 Thus, $A_{p_k}$ is in the following form.
 $$A_{p_{k}}=\begin{pmatrix}0&A_{p_{k-1}}&\cdots&A_{p_{k-1}}\\A_{p_{k-1}}&0&\cdots&A_{p_{k-1}}\\ \vdots&\vdots&\ddots&\vdots\\A_{p_{k-1}}&A_{p_{k-1}}&\cdots &0\end{pmatrix},$$
 From Step(3), we get a block $A_{p_{k-2}}$ inside each matrix $A_{p_{k-1}},$ which is repeating $(p_{k-1}-1)$ times in each row. Each $A_{p_{k-2}}$ corresponds to a pair of subsets $(V_{i_1i_2i_3},V_{j_1j_2j_3})$ for $i_3\neq j_3$ and $i_3,j_3\in \{0,1,\ldots,p_{k-1}-1\}.$ 
 Note that when $i_3=j_3$, the block corresponding to $(V_{i_1i_2i_3},V_{j_1j_2j_3})$ is the zero matrix.
 Thus, $A_{p_{k-1}}$ is in the following form.
 
 $$A_{p_{k-1}}=\begin{pmatrix}0&A_{p_{k-2}}&\cdots&A_{p_{k-2}}\\A_{p_{k-2}}&0&\cdots&A_{p_{k-2}}\\ \vdots&\vdots&\ddots&\vdots\\A_{p_{k-2}}&A_{p_{k-2}}&\cdots &0\end{pmatrix},$$
 and $A_{p_{k-2}}$ and $0$ are the blocks of size $\frac{n_{0}}{p_{1}p_{k}p_{k-1}}\times \frac{n_{0}}{p_{1}p_{k}p_{k-1}}.$\\
 ~\\
 Likewise we can continue splitting $A_{p_i}$ in each step till we get $A_{p_{3}},$ and thus inside $A_{p_3}$ we get $A_{p_{2}}$ which is of the following form
 $$A_{p_{2}}=(J-I)$$ and it is of size $p_{2}\times p_{2}.$ 
 \end{enumerate}
Hence we can compactly write ${\widetilde{A}_{11}}$ as the Kronecker product $(J_{p_{1}}-I_{p_{1}})\otimes A_{p_{k}},$ where $J_{p_1}$ and $I_{p_1}$ are the matrices of order $p_{1}\times p_{1}.$  In a similar way, $k-1$ tuples of the block matrices $(A_{p_{k}}, A_{p_{k-1}},\ldots, A_{p_{2}})$ can be seen as 
 \begin{align*}
 A_{p_{k}}&=(J_{p_k}-I_{p_k})\otimes A_{p_{k-1}}\\
A_{p_{k-1}}&=(J_{p_{k-1}}-I_{p_{k-1}})\otimes A_{p_{k-2}}\\
 &\vdots\\
 \A_{p_{3}}&=(J_{p_{3}}-I_{p_{3}})\otimes A_{p_{2}},
 \end{align*}
 where $J_{p_i}$ and $I_{p_{i}}$ are the matrices of order $p_{i}\times p_{i}$ ($3\leq i\leq k $). Lastly, $A_{p_{2}}=J_{p_2}-I_{p_2}$ is of order $p_{2}\times p_{2}.$ Now we can easily compute the eigenvalues by the backward substitution in the above system of equations to get the whole spectrum of ${\widetilde{A}_{11}}.$ Clearly each eigenvalue concerning $n,$ is of the form $(\pm)\frac{n}{n_0}\prod\limits_{i=1}^{k} \varphi(p_i)^{t_i},$ where $t_i\in\{0,1\}.$
 Let us illustrate the algorithm through an example.
 \begin{ex}
Let $n=45$ and thus $n_0=15$. Let $p_1=3$ and $p_2=5.$ We computed the following block form of the matrix ${\widetilde{A}_{11}}.$ We relabeled the vertices in the rows and the columns of ${\widetilde{A}_{11}}$ using the above algorithm. In each step of the algorithm we get the following sequences of cells. On following the respective labelling, we get 
\begin{enumerate}
    \item $[s],[sr^1],[sr^2],\ldots, [sr^{14}].$
    \item $[s],[sr^3],[sr^6],[sr^9],[sr^{12}],[sr],[sr^4],[sr^7],[sr^{10}],[sr^{13}],[sr^2],[sr^5],[sr^{8}],[sr^{11}],[sr^{14}].$
    \item $[s],[sr^6],[sr^{12}],[sr^3],[sr^{9}],[sr^{10}],[sr],[sr^7],[sr^{13}],[sr^4],[sr^5],[sr^{11}],[sr^2],[sr^{8}],[sr^{14}].$
\end{enumerate}
Since $n$ has only two prime divisors, thus we get only one block matrix that is, $A_{p_2}=(J-I)_{5\times 5}.$ So ${\widetilde{A}_{11}}$ can be seen as
$${\widetilde{A}_{11}}=3((J-I)_{3\times 3}\otimes A_{p_2}).$$
After expanding, we have
$${\widetilde{A}_{11}}=3\begin{pmatrix}
    0&J-I&J-I\\
     J-I&0&J-I\\
  J-I&J-I&0
 \end{pmatrix},$$
 where $0$ and $J-I$ are the matrices of order $5\times 5.$
 Therefore, the eigenvalues of $(J-I)_{3\times 3}$ are $\{-1,-1,2\}$ and $(J-I)_{5\times 5}$ are $\{-1,-1,-1,-1,4\}$. Hence the spectrum  of ${\widetilde{A}_{11}}$ is as follows $$\sigma({\widetilde{A}_{11}})=\begin{pmatrix}
    -12&-6&3&24&\\
    2&4&8&1
\end{pmatrix}.$$
\end{ex}
\subsection{Spectrum of the Laplacian Matrix}\label{sec4:SLM}
In this section, we will define the Laplacian matrix of the graph $\Ga_n$ of $D_n.$ Let $D=\text{diag}(\alpha_{1},\alpha_{2},\ldots,\alpha_{2n})$ where $\alpha_i's$ represent the vertex degrees of $\Ga_n,$ called the diagonal matrix. Let $L(n)$ denotes the Laplacian matrix of $\Ga_n$ and it is defined as $L(n)=D-A(n).$ \\
Using the results discussed in Section~\ref{sec4:AM}, we can conclude that if $\psi_{1},\ldots ,\psi_{n}$; $\lambda_{1},\ldots, \lambda_{\varphi(n)}$; $\gamma_{1},\ldots ,\gamma_{n-\varphi(n)}$ are the eigenvalues of $A_{11},A_{22},A_{33}$ respectively, then we can find the spectrum of the associated Laplacian matrices $L(\Ga_{\Om_i}).$ By the definition of Laplacian matrix, the eigenvalues of $L(\Ga_{\Om_2}), L(\Ga_{\Om_1})$ and $L(\Ga_{\Om_3})$  are  $\varphi(n)-\psi_{j_1},$ where $1\leq j_1\leq n$; $-\lambda_{j_2},$ where $1\leq j_2\leq \varphi(n)$; $-\gamma_{j_3},$ where $1\leq j_3\leq (n-\varphi(n))$ respectively. Thus using the concept of graph operations, we can get the whole spectrum of $L(n).$ \par 
Following the above discussion, let us first find the spectrum of $L
(\Ga_{\Om_i})$ for $i=1,2.$ We denote $\sigma_i,$ as the spectrum of $L(\Ga_{\Om_i}).$ The following table summarises the spectrum of the matrices $L(\Ga_{\Om_i})$ where $i=1,2.$
\begin{center}
 \begin{tabular}{|c|c|}
  \hline
 Spectrum of $L(\Ga_{\Om_1})={0-A_{22}}$
 &\quad\quad Spectrum of $L(\Ga_{\Om_2})={{\varphi}(n)I-A_{11}}$ \\
  \hline
  \hline
  &\\
 $\quad\quad\sigma_1=\begin{pmatrix}
    0\\
    \varphi(n)
\end{pmatrix}$
& $\sigma_2=
\begin{pmatrix}
    \varphi(n)-\mu(d_1)\frac{\varphi(n)}{\varphi(d_1)}&\cdots&\varphi(n)-\mu(d_{\tau(n_0)})\frac{\varphi(n)}{\varphi(d_{\tau(n_0)})}&\varphi(n)\\
   1&\cdots&\varphi(d_{\tau(n_0)})& n-n_0 
\end{pmatrix}$\\
&\\
\hline
\end{tabular}
  \end{center}
  \vglue 1mm
Note that $\mu(d_1)\frac{\varphi(n)}{\varphi(d_1)}=\varphi(n).$ Thus, $\varphi(n)-\mu(d_1)\frac{\varphi(n_0)}{\varphi(d_1)}=0.$ It is a well-known result that the multiplicity of the eigenvalue 0 of the Laplacian matrix equals the number of connected components in the graph. Since $\Ga_{\Om_2}$ is connected, thus the multiplicity of 0 is 1. Using the spectral Table~\ref{tab:2} of the graph operations, the spectrum  of ${L(\Ga_{\Om_1}{\vee}\Ga_{\Om_2})},$ denoted by $\sigma$ is given by $$\sigma=
    \begin{pmatrix}
   0& 2\varphi(n)-\mu(d_2)\frac{\varphi(n)}{\varphi(d_2)}&\cdots&2\varphi(n)-\mu(d_{\tau(n_0)})\frac{\varphi(n)}{\varphi(d_{\tau(n_0)})}&2\varphi(n)&n&n+\varphi(n)\\
   1&\varphi(d_2)&\cdots&\varphi(d_{\tau(n_0)})&n-n_0&\varphi(n)-1&1
\end{pmatrix}.$$
It is known that the eigenvalues of the Laplacian of the union of two graphs are the union of the eigenvalues of both graphs. Using this fact and $\Ga_{\Om_3}$ is $0$- regular graph of size  $n-\varphi(n)$, we concluded the spectrum of the Laplacian matrix $L(n)$ in the following theorem.
\begin{theorem}
    Let $\Ga_n$ be the generating graph of $D_n.$ Then the spectrum of $L(n),$ denoted by $\sigma(n)$ is given by
    $$\sigma(n)=
    \begin{pmatrix}
   0& 2\varphi(n)-\mu(d_2)\frac{\varphi(n)}{\varphi(d_2)}&\cdots&2\varphi(n)-\mu(d_{\tau(n_0)})\frac{\varphi(n)}{\varphi(d_{\tau(n_0)})}&2\varphi(n)&n&n+\varphi(n)\\
   n-\varphi(n)+1&\varphi(d_2)&\cdots&\varphi(d_{\tau(n_0)})&n-n_0&\varphi(n)-1&1
\end{pmatrix}.$$
\end{theorem}
\begin{cor}
    For $n\in \N,$ the energy of the Laplacian matrix $L(n)$ is given by $$LE(n)=3n\varphi(n).$$
\end{cor}
\begin{proof}
    It is known that the Laplacian matrix of a graph is a positive semidefinite matrix, therfore all the eigenvalues are non-negative. Thus, the Laplacian energy is the sum of its eigenvalues.  To compute the energy of $L(n)$, we sum all its eigenvalues. Note that $\tau(n_0)=2^k$  where $k$ is the number of distinct prime divisors of $n.$
    Hence we get the following expression.
    \begin{align*}
     LE(n)&=3n\varphi(n)-\varphi(n)-\sum\limits_{i=2}^{2^k}\mu(d_i)\varphi(d_i)\frac{\varphi(n)}{\varphi(d_i)}\\
        &=3n\varphi(n)-\varphi(n)-\sum\limits_{i=2}^{2^k}\mu(d_i){\varphi(n)}.
         \end{align*}
         It is easy to check that $\sum\limits_{i=2}^{2^k}\mu(d_i){\varphi(n)}=-\varphi(n).$
         Therefore,
     $$LE(n)=3n\varphi(n).$$
    Hence the result follows.
\end{proof}
\section{Topological Indices of $\Delta_n$}\label{sec5:top}
In this section, we first recall some definitions which we will use throughout the section. The {\em topological indices} are useful invariants to study the structural properties of molecular graphs which are connected. We will compute some of the topological indices for the graph $\Delta_n.$ Let us first make some observations concerning the distance between the pair of vertices and the degree of the vertices in $\Delta_n.$\par
Let $\Gl$ be a connected graph. The distance between two vertices $v_i$ and $v_j$ of $\mathcal{G}$ is defined to be the length of the shortest path, denoted by $d(v_i,v_j)$. Here we consider unordered pair of vertices unless specified. For $n\in \N$, the graph $\Delta_n$ has the vertex set $S(D_n).$ Recall that the sum of the degree of vertices is equal to twice the number of edges in the graph. Using this fact, the sum of the vertex degrees of $\Delta_n$ is $3n\varphi(n)$ (see Equation~\ref{pre:sum}). However, the sum of  squares of the vertex degrees is given by 
$$\sum\limits_{v_i\in S(D_n)}\deg(v_i)^2=n^2\varphi(n)+n(2\varphi(n))^2.$$ Using Equation~\ref{eq:nbd}, we have
\begin{center}
    $d(v_i,v_j)=\begin{dcases}2 &\quad v_i,v_j\in \Om_1; \quad v_i=sr^i,v_j=sr^j \in\Om_2 \text{ and } v_j\notin B_i\\
    1 &\quad v_i=sr^i,v_j=sr^j\in \Om_2 \text{ and } v_j\in B_i.
    \end{dcases}
    $
\end{center}
Then, the sum of all the distance between the unordered pair of distinct vertices which are the members of $\Om_1$ is given by
\begin{align*}
    \sum\limits_{v_i, v_j\in \Om_1}d(v_i,v_j)=\frac{2\varphi(n)(\varphi(n)-1)}{2}.
\end{align*}
The sum of all the distances between the unordered pair of distinct vertices which are the elements of $\Om_1$ and $\Om_2$ is given by
\begin{align*}
    \sum\limits_{v_i\in \Om_1; v_j\in \Om_2}d(v_i,v_j)=n\varphi(n).
\end{align*}
However, the sum of all the distances between the unordered pair of distinct vertices which are the elements of $\Om_2$ is given by
\begin{align*}
    \sum\limits_{v_i,v_j\in \Om_2}d(v_i,v_j)=\frac{n\varphi(n)+2n(n-\varphi(n)-1)}{2}.
\end{align*}
From the above observations, one can determine the degree and distance based topological indices like the {\em Wiener} index $W(\Gl)$, the {\em hyper-Wiener} index $WW(\Gl)$, the {\em Zagreb} first $M_1(\Gl),$ {\em Zagreb} second $M_2(\Gl)$ indices, the {\em Schultz} index $\text{MTI}(\Gl)$ and the {\em Gutman} index $\text{Gut}(\Gl)$ of a graph $\Gl.$ For more details, one can refer to~ \cite{wiener,Hwiener,zagreb,schultz,gutman}. Now in the following theorem we will compute the expression of these indices for $\Gl=\Delta_n,$ in general for any $n\in \N.$
\begin{pro}
Let $\Delta_n$ be the graph formed by removing all the isolated vertices of the graph $\Ga_n$ of $\D_n$ for any $n\in \N$. Then we have 
\begin{table}[!ht]\label{Indices}
\centering
 \begin{tabular}{|c|c|}
 \hline
 Formula of topological indices &  Expression for $\Gl=\Delta_n$\\
 \hline
 \hline
 &\\
$W(\Gl)=\sum\limits_{1\leq i\neq j\leq n}d(v_i,v_j)$&$W(\Delta_n)=\varphi(n)\left(\varphi(n)+\frac {(n-2)}{2}\right)+n(n-1)$\\
&\\
\hline
&\\
$WW(\Gl)=\frac{1}{2} W(\Gl)+\frac{1}{2}\sum\limits_{\{u,v\}\subseteq V(\Gl)} d(u,v)^2$&$WW(\Delta_n)=\frac{3}{2}\left(\varphi(n)^2-\varphi(n)+n(n-1)\right)$\\
&\\
\hline
&\\
$M_1(\Gl)=\sum\limits_{v\in V(\Gl)}\deg(v)^2$
&$M_1(\Delta_n)=n^2\varphi(n)+4n\varphi(n)^2$\\
&\\
\hline
&\\
$M_2(\Gl)=\sum\limits_{uv\in E(\Gl)}\deg(u)deg(v)$
&$M_2(\Delta_n)=2n\varphi(n)^2(n+\varphi(n))$\\
&\\
\hline
&\\
$\text{MTI}(\Gl)=\sum\limits_{\{u,v\}\subseteq V(\Gl)} d(u,v)[\deg(u)+\deg(v)]$
&$\text{MTI}(\Delta_n)=n\varphi(n)\left(2\varphi(n)+5n-6\right)$\\
&\\
\hline
&\\
$\text{Gut}(\Gl)=\sum\limits_{\{u,v\}\subseteq V(\Gl)}d(u,v)[\deg(u)\deg(v)]$
&$\text{Gut}(\Delta_n)= n\varphi(n)((7n-6)\varphi(n)-n)$\\
&\\
\hline
  \end{tabular}
\caption{Topological indices of $\Delta_n.$}
\end{table}
\end{pro}
In the following corollary, we compute indices for some special cases of $n.$
\begin{cor}
    If $n=p,\,2^\alpha$, where $\alpha\geq 1$ and $p$ is a prime, then the following table summarises the above defined topological indices for $\Delta_n$.
    \newpage

   \begin{table}[!ht]
   \centering
 \begin{tabular}{|c|c|c|c|}
  \hline
{Indices} & $\quad \quad {n=p}$& $\quad \quad{n=2^\alpha;\alpha\geq 1}$\\
  \hline
  \hline
  ${W(\Delta_n)}$& $\left(\frac{5p-4}{2}\right)(p-1)$&$3.2^{\alpha-1}(2^\alpha-1)$\\
\hline
${WW(\Delta_n)}$&$3(p-1)^2$&$\frac{3}{2}(2^{\alpha-1}(2^{\alpha-1}-1)+2^{\alpha}(2^{\alpha}-1))$\\
\hline
${M_1(\Delta_n)}$&$p(p-1)(5p-4)$&$3.2^{3\alpha-1}$\\
\hline
${M_2(\Delta_n)}$&$2p(p-1)^2(2p-1)$& $3.2^{2(2\alpha-1)}$\\
\hline
${\text{MTI}(\Delta_n)}$&$p(p-1)(7p-8)$&$3. 2^\alpha(2^\alpha-1)$\\
\hline
${\text{Gut}(\Delta_n)}$&$p(p-1)(7(p-1)^2-1)$& $2^{2\alpha}(7.2^{\alpha-1}-4)$\\
\hline
\end{tabular}
  \caption{Topological indices of $\Delta_p$ and $\Delta_{2^{\alpha}}.$}
  \end{table}
  \vglue 1mm
    
\end{cor}


\begin{thebibliography}{10}
\bibitem{spectra}
{\sc A.~Abdussakir, R.~R. Elvierayani, and M.~Nafisah}, {\em On the spectra of
  commuting and non commuting graph on dihedral group}, CAUCHY: Jurnal
  Matematika Murni dan Aplikasi, 4 (2017), pp.~176--182.

\bibitem{zindex}
{\sc M.~R. Ahmadi and R.~Jahani-Nezhad}, {\em Energy and wiener index of
  zero-divisor graphs}, Iranian Journal of Mathematical Chemistry, 2 (2011),
  pp.~45--51.

\bibitem{conradsubgroup2}
{\sc K.~CONRAD}, {\em Subgroup series ii}.
\newblock
  \url{https://kconrad.math.uconn.edu/blurbs/grouptheory/subgpseries2.pdf}.

\bibitem{crestani2013generating}
{\sc E.~Crestani and A.~Lucchini}, {\em The generating graph of finite soluble
  groups}, Israel Journal of Mathematics, 198 (2013), pp.~63--74.

\bibitem{MR1987022}
{\sc E.~Detomi and A.~Lucchini}, {\em Crowns and factorization of the
  probabilistic zeta function of a finite group}, J. Algebra, 265 (2003),
  pp.~\,651--668.

\bibitem{godsil}
{\sc C.~Godsil and G.~F. Royle}, {\em Algebraic graph theory}, Springer Science
  \& Business Media, 2001.

\bibitem{Guralnick}
{\sc R.~Guralnick and W.~Kantor}, {\em Probabilistic generation of finite
  simple groups}, Journal of Algebra, 234 (2000), pp.~\,743--792.

\bibitem{gutman}
{\sc I.~Gutman}, {\em Selected properties of the schultz molecular topological
  index}, Journal of Chemical Information and Computer Sciences, 34 (1994),
  pp.~1087--1089.

\bibitem{zagreb}
{\sc I.~Gutman and N.~Trinajsti{\'c}}, {\em Graph theory and molecular
  orbitals. total $\varphi$-electron energy of alternant hydrocarbons},
  Chemical physics letters, 17 (1972), pp.~535--538.

\bibitem{Hwiener}
{\sc D.~J. Klein, I.~Lukovits, and I.~Gutman}, {\em On the definition of the
  hyper-wiener index for cycle-containing structures}, Journal of chemical
  information and computer sciences, 35 (1995), pp.~50--52.

\bibitem{Lieback}
{\sc M.~W. Liebeck and A.~Shalev}, {\em Simple groups, probabilistic methods,
  and a conjecture of kantor and lubotzky}, Journal of Algebra, 184 (1996),
  pp.~31--57.

\bibitem{lucchini2017diameter}
{\sc A.~Lucchini}, {\em The diameter of the generating graph of a finite
  soluble group}, Journal of Algebra, 492 (2017), pp.~28--43.

\bibitem{planar}
{\sc A.~Lucchini}, {\em Finite groups with planar generating graph}, Australas.
  J. Combin., 76 (2020), pp.~220--225.

\bibitem{MR2515391}
{\sc A.~Lucchini and A.~Mar\'{o}ti}, {\em On the clique number of the
  generating graph of a finite group}, Proc. Amer. Math. Soc., 137 (2009),
  pp.~\,3207--3217.

\bibitem{lucchini2009clique}
{\sc A.~Lucchini and A.~Mar{\'o}ti}, {\em On the clique number of the
  generating graph of a finite group}, Proceedings of the American Mathematical
  Society, 137 (2009), pp.~3207--3217.

\bibitem{MR2816431}
{\sc A.~Lucchini and A.~Mar\'{o}ti}, {\em Some results and questions related to
  the generating graph of a finite group}, in Ischia group theory 2008, World
  Sci. Publ., Hackensack, NJ, 2009, pp.~183--208.

\bibitem{Nindex}
{\sc M.~Mirzargar and A.~R. Ashrafi}, {\em Some distance-based topological
  indices of a non-commuting graph}, Hacet. J. Math. Stat., 41 (2012),
  pp.~515--526.

\bibitem{Ncindex}
{\sc N.~H. Sarmin, N.~I. Alimon, and A.~Erfanian}, {\em Topological indices of
  the non-commuting graph for generalised quaternion group}, Bulletin of the
  Malaysian Mathematical Sciences Society, 43 (2020), pp.~3361--3367.

\bibitem{schultz}
{\sc H.~P. Schultz}, {\em Topological organic chemistry. 1. graph theory and
  topological indices of alkanes}, Journal of Chemical Information and Computer
  Sciences, 29 (1989), pp.~227--228.

\bibitem{west2001introduction}
{\sc D.~B. West et~al.}, {\em Introduction to graph theory}, vol.~2, Prentice
  hall Upper Saddle River, 2001.

\bibitem{wiener}
{\sc H.~Wiener}, {\em Structural determination of paraffin boiling points},
  Journal of the American chemical society, 69 (1947), pp.~17--20.
\end{thebibliography}
\end{document}